\documentclass[a4paper,12pt]{article}
\usepackage{graphicx,subfigure}

\usepackage{mymacros}

\newcommand{\bi}{\bsi}

\newcommand{\bx}{\bsx}

\newcommand{\PSL}{PSL}

\newcommand{\pic}{\operatorname{pic}}
\newcommand{\Sch}{\scS ch}
\newcommand{\Grpd}{\scG rpd}
\newcommand{\dgend}{\operatorname{end}}
\newcommand{\bb}{\bsb}
\newcommand{\be}{\bse}
\newcommand{\scHom}{\mathop{\scH om}\nolimits}

\newcommand{\fr}{\mathrm{fr}}

\newcommand{\subsect}
{\ifnum \value{subsection}<1
\refstepcounter{subsection}
\noindent
{\bf \thesubsection.} \else
\ \\
\refstepcounter{subsection}
\noindent
{\bf \thesubsection.} \fi}

\title{Weighted projective lines as fine moduli spaces
of quiver representations}
\author{Tarig Abdelgadir and Kazushi Ueda}
\date{}
\begin{document}

\maketitle

\begin{abstract}
We describe weighted projective lines
in the sense of Geigle and Lenzing
by a moduli problem on the canonical algebra of Ringel.
We then go on to study generators
of the derived categories of coherent sheaves
on the total spaces of their canonical bundles, and
show that they are rarely tilting.
We also give a moduli construction for these total spaces
for weighted projective lines with three orbifold points.
\end{abstract}

\section{Introduction}
 \label{sc:introduction}

{\em Weighted projective lines}
in the sense of Geigle and Lenzing
\cite{Geigle-Lenzing_WPC}
are one-dimensional smooth rational Deligne-Mumford stacks
without generic stabilizers.
They have fascinating connections
with many branches of mathematics,
such as representation theory of finite-dimensional algebras,
singularity theory,
hypergeometric functions, and
automorphic forms.
The reader can consult the review article \cite{MR2931898} and references therein
for more on such interactions.

Let $Q$ be the quiver
\vspace{3mm}
\begin{equation} \label{eq:quiver0}
\begin{psmatrix}[colsep=20mm,rowsep=10mm]
 & v_{1,1} & v_{1,2} & \cdots & v_{1,p_1-1} \\
 & v_{2,1} & v_{2,2} & \cdots & v_{2,p_2-1} \\[-6mm]
  v_0  & & & & & v_1 \\[-8mm]
 & \vdots & \vdots & & \vdots & \\[-4mm]
 & v_{n,1} & v_{n,2} & \cdots & v_{n,p_n-1}
\end{psmatrix}
\psset{nodesep=5pt,arrows=->,shortput=nab}
\ncdiag[angleA=50,angleB=200,arm=0]{3,1}{1,2}^{a_{1,0}}
\ncline{1,2}{1,3}^{a_{1,1}}
\ncline{1,3}{1,4}^{a_{1,2}}
\ncline{1,4}{1,5}^{a_{1,p_1-2}}
\ncdiag[angleA=-10,angleB=140,arm=0]{1,5}{3,6}^{a_{1,{p_1-1}}}
\ncdiag[angleA=10,angleB=190,arm=0]{3,1}{2,2}_{a_{2,0}}
\ncline{2,2}{2,3}^{a_{2,1}}
\ncline{2,3}{2,4}^{a_{2,2}}
\ncline{2,4}{2,5}^{a_{2,p_2-2}}
\ncdiag[angleA=-10,angleB=170,arm=0]{2,5}{3,6}_{a_{2,{p_2-1}}}
\ncdiag[angleA=-30,angleB=160,arm=0]{3,1}{5,2}_{a_{n,0}}
\ncline{5,2}{5,3}^{a_{n,1}}
\ncline{5,3}{5,4}^{a_{n,2}}
\ncline{5,4}{5,5}^{a_{n,p_n-2}}
\ncdiag[angleA=10,angleB=220,arm=0]{5,5}{3,6}_{a_{n,{p_n-1}}}
\end{equation}
and $I \subset \bC Q$ be the ideal of relations
generated by
\begin{equation} \label{eq:relations0}
 a_{i, p_i-1} \cdots a_{i, 0}
  - a_{2, p_2-1} \cdots a_{2, 0}
  + \lambda_i a_{1, p_1-1} \cdots a_{1, 0}, \qquad i = 3, \ldots, n.
\end{equation}
The path algebra $A := \bC Q / I$ with relations is the canonical algebra
introduced by Ringel \cite{MR774589}.
By \cite[Proposition 4.1]{Geigle-Lenzing_WPC},
there is a tilting object $E = \bigoplus_{i=1}^N E_i$ in $D^b \coh \bX$
obtained as the direct sum of line bundles $(E_i)_{i=1}^n$,
whose endomorphism algebra is isormorphic to $A$;
\begin{align} \label{eq:AE}
 A \cong \End E.
\end{align}
Then Morita theory for derived categories \cite{Bondal_RAACS,Rickard}
gives an equivalence
\begin{align} \label{eq:der_equiv0}
 D^b \coh \bX \cong D^b \module A
\end{align}
of triangulated categories.
%
%
Let $Q_0$ and $Q_1$ be the set of vertices and arrows of the quiver $Q$,
and $\bZ^{Q_0}$ be the free abelian group generated by the symbols $\be_\bi$
for $\bi \in Q_0$.
Consider the natural homomorphism
\begin{align} \label{eq:pic}
 \pic : \bZ^{Q_0} \to \Pic \bX
\end{align}
sending $\be_\bi$ to the line bundle corresponding to $\bi \in Q_0$.
The kernel of $\pic$ will be denoted by
\begin{align} \label{eq:K}
 K := \Ker (\pic).
\end{align}
In Section \ref{sc:refined},
we give an adaptation of the notion of {\em refinements} from \cite{Abdelgadir:2012}
to the current non-toric setting,
and define the moduli space $\scM_\theta(Q, I, K; \bsd)$
of $K$-refined representations of the quiver $(Q, I)$ with relations.
Here $\bsd$ is the dimension vector and $\theta$ is the stability parameter.
The first main result in this paper is the following:

\begin{theorem} \label{th:main1}
For a suitable choice $\vartheta$ of a stability parameter,
there is an isomorphism
$\bX \cong \scM_\vartheta(Q,I,K; \bsone)$
of stacks.
The pull-back of the tautological bundle on $\scM_\vartheta(Q,I,K;\bsone)$
by this isomorphism is isomorphic to $E$.
\end{theorem}

A weighted projective line $\bX$ is called
{\em spherical},
{\em Euclidean}, or
{\em hyperbolic}
if the orbifold Euler characteristic
$$
 \chi_{\mathrm{orb}}(\bX)
  = 2 - \sum_{i=1}^n \lb 1 - \frac{1}{p_i} \rb
$$
is positive, zero, or negative respectively.
Spherical weighted projective lines naturally correspond
to simply-laced Dynkin diagrams as follows:
\begin{table}[h]
$$
\begin{array}{|c|c|c|c|c|c|}
 \hline
 \text{type} & A_{p+q} & D_{n+2} & E_6 & E_7 & E_8 \\
 \hline
 \text{weight} & (1,p,q) & (2,2,n) & (2,3,3) & (2,3,4) & (2,3,5) \\
 \hline
\end{array}
$$
\caption{Spherical weights}
\label{tb:spherical_wt}
\end{table}

\noindent
When $\bX$ is spherical,
the spectrum $Y_0 = \Spec R$
of the anti-canonical ring
\begin{align*}
 R
  = \bigoplus_{k=0}^\infty H^0(\scO_\bX(-k \omega_\bX)),
\end{align*}
is isomorphic to the corresponding Kleinian singularity
\cite[Proposition 8.5]{Geigle-Lenzing_PC}.
The total space $\bK$ of the canonical bundle of $\bX$
has orbifold singularities of types $\frac{1}{p_i}(1, -1)$
at the $i$-th orbifold points in the zero-section $\bX \subset \bK$,
so that the coarse moduli space $K$ of $\bK$
is a partial resolution of $Y$
with three simple singularities
of types $A_{p_1-1}$, $A_{p_2-1}$ and $A_{p_3-1}$.
The minimal resolution $Y$ of $K$
is the minimal resolution of $Y_0$,
and the exceptional divisor is a tree of $(-2)$-curves
where three chains of $(-2)$-curves
consisting of $p_1-1$, $p_2-1$ and $p_3-1$ rational curves
is connected to the $(-2)$-curve
obtained as the strict transform
of the zero-section in $K$.
One has a derived equivalence
$$
 D^b \coh \bK \cong D^b \coh Y
$$
by Kapranov and Vasserot
\cite{Kapranov-Vasserot}.
The pull-back $\Etilde = \pi^* E$
of the tilting object $E$
by the natural projection $\pi : \bK \to \bX$
is a generator of $D^b \coh \bK$.
Recall that $\Etilde$ is said to be {\em acyclic}
if $\Hom^i(\Etilde, \Etilde) = 0$ for any $i \ne 0$.
We prove the following in Section \ref{sc:tilting}
by a case-by-case analysis through Figure \ref{tb:spherical_wt}:

\begin{theorem} \label{th:tilting}
$\Etilde$ is acyclic if and only if $\bX$ is spherical of type $A$.
\end{theorem}

The endomorphism dg algebra $\dgend \Etilde$
is quasi-isomorphic to the 2-Calabi-Yau completion
(or the derived 2-preprojective algebra)
$$
 \dgend \Etilde \cong \Pi_2(\End E)
$$
of  $\End E$
in the sense of Keller
\cite[4.1]{Keller_DCYC},
and one has an equivalence
$$
 D^b \coh \bK \cong D^b \module \Pi_2(\End E)
$$
of triangulated categories
\cite{Ballard_SLCY,Segal_ADT}.
Theorem \ref{th:tilting} shows that $\dgend \Etilde$ is quasi-isomorphic to $\End \Etilde$
if and only if $\bX$ is spherical of type $A$.

Let $\Gbar \subset PSL_2(\bC)$ be a finite Kleinian group,
so that the quotient stack $\bX = [\bP^1 / \Gbar]$
is a spherical weighted projective line.
The dg algebra $\dgend \Etilde$ is derived-equivalent
to the crossed-product algebra
$\bC[x,y] \rtimes G$,
where $G$ is the inverse image of $\Gbar$
by the projection $\SL_2(\bC) \to \PSL_2(\bC)$.
The crossed product algebra $\bC[x,y] \rtimes G$ in turn
is Morita-equivalent to the path algebra with relations
of the {\em McKay quiver} of $G$.
Kronheimer's construction
\cite{Kronheimer_ALE}
(cf. also \cite[Section 4]{Nakajima_LHSPS})
gives a fine moduli interpretation of the crepant resolution
$Y$ of $\bC^2/G$
in terms of {\em framed} representations
of the McKay quiver.

Let $\bX$ be a weighted projective line
with at most three orbifold points
and $\pi : \bK \to \bX$ be the canonical bundle.
Let further $(\Qtilde, \Itilde)$ be the quiver with relations
describing the endomorphism algebra of the vector bundle $\Etilde := \pi^* E$
on $\bK$.
In Section \ref{sc:moduli2},
we introduce the moduli stack $\scM_\theta^{\fr}(\Qtilde, \Itilde; \bsone)$
of a {\em framed refined representation} of $(\Qtilde, \Itilde)$
of dimension vector $\bsone$.
Unlike the construction of Kronheimer,
we frame arrows instead of vertices.
Although $\bC \Qtilde/\Itilde$ is not derived-equivalent to $\bK$,
we can show the following:

\begin{theorem} \label{th:main2}
For a suitable choice $\vartheta$ of a stability parameter,
there is an isomorphism
$\scM_\vartheta^{\fr}(\Qtilde, \Itilde; \bsone) \cong \bK$
of stacks
sending the tautological bundle to $\Etilde$.
\end{theorem}

It is an interesting problem to give a fine moduli interpretation
of the quotient stack $[\bC^2/G]$
in terms of quiver representations.
Although Theorem \ref{th:main2} falls short of this goal,
it gives a fine moduli interpretation of $\bK$
which sits between $Y$ and $[\bC^2 / G]$.



\ \\[-5mm]

{\em Acknowledgment} :
This work was initiated
at Max Planck Institute for Mathematics,
and completed
at Korea Institute for Advanced Study.
We thank both institutes for hospitality
and nice research environment.
We also thank the anonymous referee for suggesting a number of improvements.

\section{Weighted projective lines after Geigle and Lenzing}
 \label{sc:wpl}

Let $\bX$ be a smooth rational Deligne-Mumford stack
of dimension one without generic stabilizer.
Such a stack is called a {\em weighted projective line} by Geigle and Lenzing.
We recall some basic results from \cite{Geigle-Lenzing_WPC} in this section.

We write the orbifold points on $\bX$ as $\bx_1, \dots, \bx_n$.
The absence of generic stabilizer implies that
the stabilizer group $\Gamma_i$ at $\bx_i$
is a cyclic group
for any $i = 1, \dots, n$,
and we write its order as $p_i$.
Locally around $\bx_i$,
one can take an orbifold chart
$
 [U_i / \Gamma_i] \hookrightarrow \bX,
$
where $U_i$ is an open subscheme
of an affine line $\bA^1 = \Spec \bC[u]$ and
a generator of $\Gamma_i$ acts linearly on $u$
by a primitive $p_i$-th root of unity.
The Picard group of $\bX$ is generated
by $\scO_\bX(\vecx_i)$ for $i = 1, \dots, n$,
where $\scO_\bX(\vecx_i)$ is the dual of $\scO_\bX(-\vecx_i)$
defined as the kernel of the natural morphism
$
 \scO_\bX \to \scO_{\bx_i}
$
to the skyscraper sheaf
$
 \scO_{\bx_i}
  = \big[ \big( \Spec \bC[u]/(u) \big) \big/ \Gamma_i \big];
$
$$
 0 \to \scO_\bX(-\vecx_i) \to \scO_\bX \to \scO_{\bx_i} \to 0.
$$
Define $\scO_\bX(\vecc)$ as the line bundle $\scO_\bX(x)$,
which does not depend on the choice of a point
$x \in \bX \setminus \{ \bx_1, \dots, \bx_n \}$.
One has relations
$$
 \scO_\bX(p_i \vecx_i) = \scO_\bX(\vecc), \quad i = 1, \dots, n,
$$
and the Picard group of $\bX$ is given by
$$
 L
  := \Pic \bX
  = \bZ \vecx_1 \oplus \dots \bZ \vecx_n \oplus \bZ \vecc
      / (p_1 \vecx_1 - \vecc, \dots, p_n \vecx_n - \vecc).
$$
Choose a global coordinate
on the coarse moduli space $X \cong \bP^1$ so that
the images $z_1, \dots, z_n \in X$
of $\bx_1, \ldots, \bx_n \in \bX$ are given in this coordinate
by
$\lambda_1 = \infty$,
$\lambda_2 = 0$,
$\lambda_3 = 1$, and
$\lambda_4, \dots, \lambda_n \in \bP^1 \setminus \{ 0, 1, \infty \}$.
Then the {\em total coordinate ring}
(also known as the {\em Cox ring}) of $\bX$ is
the $L$-graded ring given by
\begin{align} \label{eq:Cox_ring}
 S_\bX
  = \bigoplus_{\veck \in L} H^0(\scO_\bX(\veck))
  = k[X_1, X_2, \dots, X_n] \, \big/
       \lb X_i^{p_i} - X_2^{p_2} + \lambda_i X_1^{p_1} \rb_{i=2}^n,
\end{align}
such that $\deg X_i = \vecx_i$ for $i = 1, \dots, n$,
and $\bX$ is recovered as the quotient stack
\begin{align} \label{eq:bX}
 \bX = \big[ (\Spec S_\bX \setminus \bszero) \big/ G \big]
\end{align}
where $G$ is the affine algebraic group
$
 G = \Spec \bC[L].
$
The graded ring $S_\bX$ is Gorenstein with parameter
$
 \vecomega = (n-2)-\sum_{i=1}^n \vecx_i,
$
and Serre duality on $\bX$ is given by
$$
 \Ext^1(\scE, \scF)
  \cong \Hom(\scF, \scE \otimes \scO_\bX(\vecomega))^\vee
$$
for any coherent sheaves $\scE$ and $\scF$.

Recall that an object $F$ in the derived category $D^b \coh Y$
of a smooth Deligne-Mumford stack $Y$
{\em classically generates} $D^b \coh Y$
if $D^b \coh Y$ is the smallest thick triangulated subcategory of $D^b(Y)$ containing $F$.
The object $F$ is {\em acyclic}
if $\Hom(F, F[i]) = 0$ for any $i \ne 0$.
The object $E$ is called a {\em tilting object}
if it is an acyclic object
which classically generates $D^b \coh Y$.
The direct sum $E = \bigoplus_{i=1}^N E_i$
of the sequence
\begin{align}  \label{eq:EC}
 (E_1, \dots, E_N)
  = (\scO, \scO(\vecx_1), \dots, \scO((p_1-1) \vecx_1)&,
  \scO(\vecx_2), \dots, \scO((p_2-1) \vecx_2), \nonumber \\
  &\dots, \scO(\vecx_n), \dots, \scO((p_n-1) \vecx_n),
  \scO(\vecc))
\end{align}
of line bundles on $\bX$
is a tilting object in $D^b \coh \bX$
by \cite[Proposition 4.1]{Geigle-Lenzing_WPC}.
The endomorphism algebra $\End E$ is isomorphic
to the path algebra $A$ of the quiver $(Q, I)$ with relations
defined in \eqref{eq:quiver0} and \eqref{eq:relations0}.
The line bundles $(E_1, \ldots, E_N)$ on $\bX$ correspond to the vertices
$
 (v_0, v_{1,1}, \ldots, v_{1,p_1-1}, \ldots, v_{n,1}, \ldots, v_{n, p_n-1}, v_1)
$
in such a way that the isomorphism \eqref{eq:AE}
sends the idempotent associated with the vertex
to the identity morphism of the corresponding line bundle.

\section{The moduli space of refined representations}
 \label{sc:refined}


Let $B$ be the subset of $\bZ^{Q_0}$
consisting of
$
 \bb_{i,j}
  := \be_{i,j+1} - \be_{i,j} - \be_{i,1} + \be_0
$
and
$
 \bb_{0} := \be_{0}.
$
This set forms a basis of the group
$
 K
$
defined in \eqref{eq:K}.
We regard $b \in B$
as an element of $\Hom(\bZ^{Q_0}, \bZ)$
by the isomorphism
$\Hom(\bZ^{Q_0}, \bZ) \cong \bZ^{Q_0}$
coming from the standard basis $\{ \be_{\bi} \}_{\bi \in Q_0}$
of $\bZ^{Q_0}$.
The following notion of {\em refinements}
is an adaptation from \cite{Abdelgadir:2012}:

\begin{definition} \label{df:admissibility}
A {\em $K$-refined representation} of $(Q,I)$
is a finite-dimensional representation\\
$((W_\bi)_{\bi \in Q_0}, (w_a)_{a \in Q_1})$
of $Q$ satisfying the relations $I$
together with non-zero elements
\begin{align*}
 f_{i,j} \in \Hom( \Hom(W_0, W_{i,1}), \Hom(W_{i,j}, W_{i,j+1})),
  \quad i = 1, \ldots, n, \  j = 1, \ldots, p_i - 1,
\end{align*}
satisfying
\begin{equation} \label{eq:refinement1}
 f_{i,j} (w_{i,0}) = w_{i,j}, \qquad
 i = 1, \ldots, n, \ 
 j= 1, \ldots p_i-1,
\end{equation}
and a non-zero element $g \in W_0$.
\end{definition}
Note that the vector space
$
 \Hom( \Hom(W_0, W_{i,1}), \Hom(W_{i,j}, W_{i,j+1}))
$
containing $f_{i,j}$ can be written as
$
 W^{\bb_{i,j}} :=
  W_{i,j+1} \otimes W_{i,j}^\vee \otimes W_{i,1}^\vee \otimes W_0.
$
Although Definition \ref{df:admissibility} is given in terms of the basis $B$ of $K$,
one can easily see that the notion of refined representation
is independent of the choice of $B$
and depends only on $K$.
The arrows $a_{i,j}$ in the quiver $Q$
naturally correspond to elements of the Cox ring $S_\bX$ of $\bX$,
and the role of the elements $f_{i,j} \in W^{\otimes \bb_{i,j}}$ is
to fix isomorphisms between arrows
corresponding to the same elements of $S_\bX$.
On the other hand,
the non-zero element $g \in W_0$ rigidifies the automorphism
given by an overall scalar multiplication.

\begin{definition} \label{df:stability}
For a filtration
\begin{equation*}
 \bsW_\bullet =  \lb
  0 \subsetneq \bsW_1
   \subset \ldots \subset \bsW_{k-1}
   \subsetneq \bsW_k=\bsW
 \rb.
\end{equation*}
of representations of $Q$
and an element
$
 \theta \in \Hom(\bZ^{Q_0}, \bZ),
$
we set
$$
 \theta(\bsW_\bullet)= \sum_{j=1}^{k-1} \theta(\dim \bsW_j)
$$
where $\dim \bsW_j \in \bZ^{Q_0}$.
The filtration $\bsW_\bullet$ is said to be {\em proper}
if at least one term is a nonzero proper subspace of $\bsW$.
A refined representation $\bsW$ of dimension vector $\bsone$
is {\em $\theta$-stable}
if $\theta(\bsW) = 0$ and
$\theta(\bsW_\bullet) > 0$
for every proper filtration $\bsW_\bullet$
of refined representations
satisfying $b(\bsW_\bullet) = 0$
for all $b \in B$.
\end{definition}


A stack can be defined either as a category
or as a 2-functor.
The first point of view is technically simpler
due to the lack of need
to introduce the notion of 2-categories,
and the second point of view is amenable to generalization
to higher and derived stacks.

From the first point of view,
our moduli problem can be formulated as follows:

\begin{definition} \label{df:M}
The category $\scM_{\theta} = \scM_\theta(Q,I,K;\bsone)$
of flat families of $\theta$-stable $K$-refined representations of $(Q, I)$
with dimension vector $\bsone$
is defined as follows:
\begin{itemize}
 \item
An object is a collection
$\scW = (S, (\scW_\bi)_{\bi \in Q_0}, (w_a)_{a \in Q_1}, (f_{i,j})_{i,j}, g)$
of
\begin{itemize}
 \item
a scheme $S$,
 \item
invertible $\scO_S$-modules $\scW_\bi$ for each $\bi \in Q_0$,
 \item
morphisms $w_a : \scW_{s(a)} \to \scW_{t(a)}$
of $\scO_S$-modules satisfying the relations \eqref{eq:relations0}
for each $a \in Q_1$,
 \item
non-zero morphisms
$
 f_{i,j} : \scHom_{\scO_S}(\scW_0, \scW_{i,1})
  \to \scHom_{\scO_S}(\scW_{i,j}, \scW_{i,j+1})
$
of $\scO_S$-modules
for $i = 1, \ldots, n$ and $j = 0, \ldots, p_i - 1$
such that $H^0(f_{i,j})$ satisfy the relations \eqref{eq:refinement1},
and
 \item
an isomorphism $g : \scO_S \simto \scW_0$
\end{itemize}
such that the refined representation
determined by the above data is
$\theta$-stable
at every point $s \in S$.
 \item
A morphism from a family $\scW$ to another family $\scW'$
is a morphism $\varphi : S \to S'$ of schemes
together with isomorphisms
$$
 \gamma_\bi : \scW_\bi \simto \varphi^* \scW'_\bi
$$
of $\scO_S$-modules
for all $\bi \in Q_0$
satisfying
\begin{itemize}
 \item
$
 \gamma_{t(a)} \circ w_a \circ \gamma_{s(a)}^{-1} = \varphi^*w_a'
$
for all $a \in Q_1$,
 \item
$
 \gamma_{i,j} \circ f_{i,j} = \varphi^* f'_{i,j}
$
for all $i, j$, and
 \item
$\gamma_0 \circ g = \varphi^*g'$,
\end{itemize}
where
$
 \gamma_{i,j} :
  \scW_0 \otimes \scW_{i,1}^\vee \otimes \scW_{i,j}^\vee \otimes \scW_{i,j+1}
   \to
  \varphi^*\scW_0'\otimes \varphi^*\scW_{i,1}'^\vee \otimes \varphi^*\scW_{i,j}'^\vee \otimes \varphi^*\scW'_{i,j+1}
$
is the morphism induced by $\gamma_\bi$.
\end{itemize}
The forgetful functor
$p : \scM_\theta \to \Sch$
sending $\scW$ to $S$ makes $\scM_\theta$
into a category fibered in groupoid.
\end{definition}

From the second point of view,
the moduli problem is formulated in terms of a 2-functor
$\scM_\theta : \Sch \to \Grpd$
from the category of schemes
(considered as a 2-category
by regarding each scheme as a category with one object
and one morphism)
to the 2-category of groupoids,
which sends a scheme $S$ to the groupoid
consisting of families over $S$.

We will take the first point of view in this paper,
and prove Theorem \ref{th:main1}
along the traditional path of first rigidifying the moduli problem
and then quotienting out by equivalences.

\begin{proof}[Proof of Theorem \ref{th:main1}]

By choosing a basis of $W_\bi$ for every $\bi \in Q_0$,
a refined representation is represented
by a point in the subvaritey $Z$ of
$$
 \scR(Q, K)
  :=  \prod_{a \in Q_1} \End(\bC) \times \prod_{b \in B} \Aut(\bC)
  \cong \bA^{Q_1} \times \bG_m^{B}
$$
cut-out by the relations
\eqref{eq:relations0} and \eqref{eq:refinement1}.
Two elements of $Z$ define isomorphic refined representations
if and only if they lie in the same orbit of the group
$
 \bG_m^{Q_0}
$
under the change-of-basis action.

Given a family $\scW$ of refined representations,
one can choose a sufficiently small open subset $U$
around each point $s \in S$ on the base scheme
such that the invertible sheaves $\scW_i$ are trivial.
By choosing a trivialization,
one obtains a morphism from $U$ to $Z$.
Although these morphisms do not necessarily glue together
to give a morphism $S \to Z$,
the failure of the gluing comes from the action of $\bG_m^{Q_0}$,
so that one can replace $S$
with a principal $\bG_m^{Q_0}$-bundle $T$,
and obtain a morphism $T \to Z$.
This shows that the category $\scM(Q,I,K;\bsone)$
of flat families of not necessarily stable refined representations
of dimension vector $\bsone$ is a stack
isomorphic to the quotient stack $[Z / \bG_m^{Q_0}]$.

Define
$
 \vartheta \in \Hom(\bZ^{Q_0}, \bZ)
$
by
\begin{equation} \label{eq:vartheta}
 \vartheta (\bsW)
  := - |Q_0| \dim W_0 + \sum_{\bi \in Q_0} \dim W_\bi
\end{equation}
and let $Z^s$ denote the open subvariety of $\vartheta$-stable refined representations of $(Q, I)$.
Then one has an isomorphism
\begin{align} \label{eq:isom1}
 \scM_\vartheta(Q,I,K;\bsone) \cong [Z^s/\bG_m^{Q_0}] 
\end{align}
of stacks.
Note that every point in $Z \subset \scR (Q, K)$ lies in the orbit of a point
whose $\bG_m^B$-component is given by $(1, \dots, 1)$, and
the unstable locus consists just of the origin $\bszero \in Z$.
Then the relations \eqref{eq:relations0} give
$$
 Z^s \cap (\bA^{Q_1} \times (1,\ldots,1))
  \cong \Spec(S_\bX) \setminus \bszero.
$$
Let $H$ be the subgroup of $\bG_m^{Q_0}$ fixing the $\bG_m^B$-component,
so that
\begin{align} \label{eq:isom2}
 [Z^s / \bG_m^{Q_0}] \cong [ (\Spec S_\bX \setminus \bszero) / H].
\end{align}
The group of characters of the subgroup $H \subset \bG_m^{Q_0}$ is given by $\bZ^{Q_0} / K$,
which is isomorphic to $\Pic \bX$ by the definition \eqref{eq:K} of $K$.
This shows that $H$ is isomorphic to $G$;
\begin{align} \label{eq:isom3}
 H \cong G := \Spec \bC[L].
\end{align}
By combining \eqref{eq:isom1}, \eqref{eq:isom2}, \eqref{eq:isom3} and \eqref{eq:bX},
one obtains an isomorphism
$$
 \scM_\vartheta(Q,I,K;\bsone) \cong \bX
$$
of stacks.

The tautological line bundles are given
by the images of the standard basis elements of $\bZ^{Q_0}$
under the map $\pic : \bZ^{Q_0} \to \Pic \bX$.
This gives the line bundles \eqref{eq:EC} by definition,
and Theorem \ref{th:main1} is proved.
\end{proof}

\begin{remark}
The group $\Hom(\bZ^{Q_0}, \bZ)$
containing the stability parameter $\vartheta$
can naturally be identified
with the group of characters of $\bG_m^{Q_0}$,
and one can show that the stability condition
comes from the geometric invariant theory
just as in \cite[Theorem 3.5]{Abdelgadir:2012}.
\end{remark}

\section{A tilting bundle on the canonical bundle}
 \label{sc:tilting}

Let $\bK$ be the total space of the canonical bundle of $\bX$, and
$\pi : \bK \to \bX$ be the natural projection.
The canonical bundle of $\bK$ is trivial
by the adjunction formula.
Let $E = \bigoplus_{i=1}^N E_i$ be the tilting object
given in \eqref{eq:EC},
and set $\Etilde = \pi^* E$.

\begin{lemma} \label{lm:generator}
$\Etilde$ classically generates $D^b \coh \bK$.
\end{lemma}

\begin{proof}
Recall that for a smooth Deligne-Mumford stack $\frakX$,
the unbounded derived category
$D(\Qcoh \frakX)$ of quasi-coherent sheaves is compactly generated,
and the full subcategory consisting of compact objects
is equivalent to the bounded derived category $D^b \coh \frakX$.
The proof for a quasi-compact, quasi-separated schemes
can be found in \cite[Theorem 3.1.1]{Bondal-van_den_Bergh},
and a much more general result for perfect stacks is given in \cite{Ben-Zvi--Francis--Nadler_IT}.
It follows that an object $F$ of $D^b \coh \frakX$ classically generates $D^b \coh \frakX$
if and only if it {\em generates} $D(\Qcoh \frakX)$
in the sense that the right orthogonal $F^\perp$
consists of zero objects
(Ravenel and Neeman \cite{Neeman_CBK}, cf. also \cite[Theorem 2.1.2]{Bondal-van_den_Bergh}).

Now let $M$ be an object of $D(\Qcoh \bK)$
right orthogonal to $\Etilde$
in the sense that $\Hom(\Etilde, M[i]) = 0$ for all $i \in \bZ$.
The standard adjunction shows
\begin{align*}
 \Hom(\Etilde, M[i])
  = \Hom(\pi^* E, M[i])
  = \Hom(E, \pi_* M[i]),
\end{align*}
which implies $\pi_* M \cong 0$.
This implies $M \cong 0$
since $\pi$ is an affine morphism.
This shows that $\Etilde$ generates $D(\Qcoh \bK)$,
and hence classically generates $D^b \coh \bK$.
\end{proof}

Now we prove Theorem \ref{th:tilting}

\begin{proof}[Proof of Theorem \ref{th:tilting}]
First note that
\begin{align*}
 \Ext^1(\pi^* E_i, \pi^* E_j)
  &\cong \Ext^1 (E_i, \pi_* \pi^* E_j) ) \\
  &\cong \Ext^1 \lb E_i, \bigoplus_{k=0}^\infty E_j \otimes \scO_\bX(- k \vecomega) \rb \\
  &\cong \bigoplus_{k=0}^\infty H^1(E_i^\vee \otimes E_j \otimes \scO_\bX(- k \vecomega)) \\
  &\cong \bigoplus_{k=0}^\infty H^0(E_i \otimes E_j^\vee \otimes \scO_\bX((k+1) \vecomega))^\vee.
\end{align*}
By considering the case $i = j$,
one obtains $H^0(\scO_\bX((k+1) \vecomega) = 0$ for any $k \in \bZ^{\ge 0}$
as a necessary condition for acyclicity.
It is easy to see that this condition is satisfied
if and only if $\bX$ is spherical.

If $\bX$ is spherical of type $A$,
then one has $\vecomega = - \vecx_1 - \vecx_2$,
and one can easily see that
\begin{align*}
 H^0(E_i \otimes E_j^\vee \otimes \scO_\bX((k+1) \vecomega)) = 0
\end{align*}
for any $i, j \in \{ 1, \ldots, N \}$ and any $k \in \bZ^{\ge 0}$.
If $\bX$ is spherical of type $D_{n+2}$,
then one has
\begin{align*}
 H^0(\scO(\vecc) \otimes \scO^\vee \otimes \scO(2 \vecomega))
 &= H^0(\scO_\bX(3 \vecc - 2 \vecx_1 - 2 \vecx_2 - 2 \vecx_3)) \\
 &= H^0(\scO_\bX( (n-2) \vecx_3)) \\
 &= \bC \cdot X_3^{n-2},
\end{align*}
so that $\Etilde$ is not acyclic.
Similarly,
one has
$
 H^0(\scO(\vecc) \otimes \scO^\vee \otimes \scO(6 \vecomega))
  \ne 0
$
if $\bX$ is spherical of type $E_6$, $E_7$ and $E_8$.

For $\Ext^2(\pi^*E_i,\pi^*E_j)$ we have
\begin{align*}
 \Ext^2(\pi^* E_i, \pi^* E_j)
  \cong \Ext^2 (E_i, \pi_* \pi^* E_j) ) 
  \cong \Ext^2 \lb E_i, \bigoplus_{k=0}^\infty E_j \otimes \scO_\bX(- k \vecomega)\rb=0.
\end{align*} 
This concludes the proof
of Theorem \ref{th:tilting}.
\end{proof}

\section{Quivers for the canonical bundle}
The endomorphism algebra $\End \Etilde$ can be described as the path algebra
$\bC \Qtilde / \Itilde$
of a quiver $(\Qtilde, \Itilde)$ with relations.
In this section,
we describe a method to calculate the underlying quiver $\Qtilde$.
%
%
Note that $\bC \Qtilde / \Itilde$ is just an algebra and not a dg algebra,
so that it is not quasi-isomorphic to the endomorphism dg algebra $\dgend \Etilde$
if $\Etilde$ is not acyclic.

Since $\pi^* E_i$ and $\pi^* E_j$ are line bundles,
any element of $\Hom(\pi^* E_i, \pi^* E_j) \subset \End \Etilde$
comes from an element of the Cox ring
$S_\bK := \bigoplus_{\scL \in \Pic \bK} H^0(\scL)$, and
one can label each arrow of $\Qtilde$
with an element of $S_\bK$.

Note that $S_{\bK}$ is isomorphic to $S_{\bX} \otimes \bC[t]$
where $t$ is the tautological section of $\pi^*\scO_\bX(\vecomega)$,
and $\bK$ is recovered as the quotient stack
\begin{align} \label{eq:Cox_ring2}
 \bK =  [((\Spec S_\bX \setminus \bszero) \times \Spec \bC[t]) / G]
\end{align}
just as in \eqref{eq:Cox_ring}.
Here $G = \Spec \bC[L]$ acts on $\Spec \bC[t]$
by the character $\vecomega \in L$.

To obtain $\Qtilde$ from $Q$,
one first replaces the vertices $E_i$ with $\pi^* E_i$
while keeping the arrows fixed.
Then for each primitive element of
$$
 \Hom_\bX(E_i, E_j \otimes \scO_\bX(-k \vecomega))
  \subset \Hom_\bK(\pi^* E_i, \pi^* E_j)
$$
corresponding to an element
$$
 a \in H^0(E_i^\vee \otimes E_j \otimes \scO_\bX(-k \vecomega))
  \subset S_\bX,
$$
we add an arrow from the vertex $\pi^* E_i$
to the vertex $\pi^* E_j$ of this quiver,
which is labeled by $a \cdot t^k \in S_\bK$.

As an example,
let us first consider the case of $\bX = \bX_{2, 3, 4}$.
We will write $x_1$, $x_2$ and $x_3$ as $x$, $y$ and $z$ respectively.
The quiver $Q$ in this case is given by
\vspace{3mm}
\begin{equation}\vspace{3mm}
\begin{psmatrix}[colsep=5mm]
 & & & & \scO(\vecx) \\
 \scO & & \scO(\vecz) & & \scO(2 \vecz) & & \scO(3 \vecz) & & \scO(\vecc), \\
 &[mnode=none] \phantom{xxx} & & \scO(\vecy) &  & \scO(2 \vecy) & &[mnode=none] \phantom{xxx}
\end{psmatrix}
\psset{nodesep=3pt,arrows=->,shortput=nab}
\ncline{2,1}{1,5}
\ncline{1,5}{2,9}
\ncline{2,1}{2,3}
\ncline{2,3}{2,5}
\ncline{2,5}{2,7}
\ncline{2,7}{2,9}
\ncline{2,1}{3,4}
\ncline{3,4}{3,6}
\ncline{3,6}{2,9}
\end{equation}
and the underlying quiver $\Qtilde$ for $\End \Ktilde$ is given by
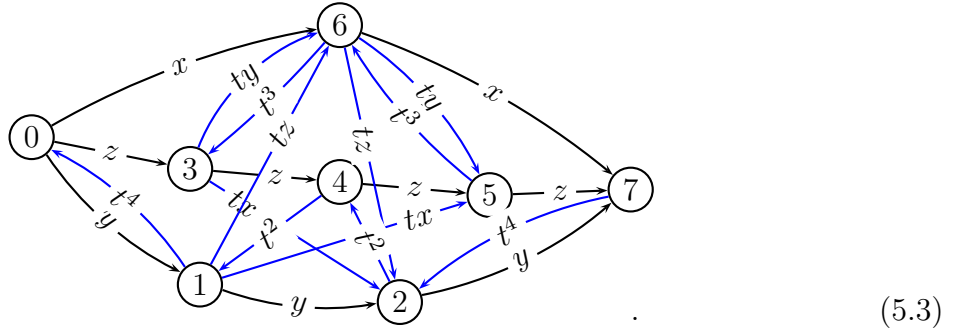
\begin{equation} \label{eq:quiver_K}
        \begin{pspicture}(-2.09,-2)(5.8,1.91)
          \cnodeput(0,0){z}{3}
          \cnodeput(1.9696,-0.1736){2z}{4} 
          \cnodeput(3.9392,-0.3473){3z}{5}
          \cnodeput(1.9696,1.9057){x}{6}
          \cnodeput(0.1290,-1.5345){y}{1}
          \cnodeput(2.7585,-1.7663){2y}{2}
          \cnodeput(-2.0842,0.4202){0}{0}
          \cnodeput(5.7942,-0.2744){c}{7}
          \psset{nodesep=0pt}
          \ncline{->}{0}{z} \lput*{:U}{$z$}
          \ncline{->}{z}{2z} \lput*{:U}(0.6){$z$}
          \ncline{->}{2z}{3z} \lput*{:U}{$z$}
          \ncline{->}{3z}{c} \lput*{:U}{$z$}
          \ncarc[arcangle=10]{->}{0}{x} \lput*{:U}{$x$}
          \ncarc[arcangle=10]{->}{x}{c} \lput*{:U}{$x$}
          \ncarc[offset=1pt,arcangle=-15]{->}{0}{y} \lput*{:U}{$y$}
          \ncarc[linecolor=blue, offset=-1pt,arcangle=15]{<-}{0}{y} \lput*{:U}{$t^4$}
          \ncarc[offset=1pt,arcangle=-15]{->}{y}{2y} \lput*{:U}{$y$}
          \ncarc[offset=1pt,arcangle=-15]{->}{2y}{c} \lput*{:U}{$y$}
          \ncarc[linecolor=blue, offset=-1pt,arcangle=15]{<-}{2y}{c} \lput*{:U}{$t^4$}
          \ncline[linecolor=blue]{->}{z}{2y} \lput*{:U}(0.2){$tx$}
          \ncline[linecolor=blue]{->}{y}{3z} \lput*{:U}(0.8){$tx$}
          \ncline[linecolor=blue]{->}{y}{x} \lput*{:U}(0.6){$tz$}
          \ncline[linecolor=blue]{->}{x}{2y} \lput*{:U}(0.4){$tz$}
          \ncarc[offset=0pt,arcangle=25,linecolor=blue]{->}{z}{x} \lput*{:U}{$ty$}
          \ncarc[arcangle=15,linecolor=blue]{->}{x}{3z} \lput*{:U}{$ty$}
          \ncline[linecolor=blue]{->}{2y}{2z} \lput*{:180}{$t^2$}
          \ncline[linecolor=blue]{->}{2z}{y} \lput*{:180}{$t^2$}
          \ncarc[offset=0pt, arcangle=5, linecolor=blue]{->}{x}{z} \lput*{:180}{$t^3$}
          \ncarc[arcangle=10,linecolor=blue]{->}{3z}{x} \lput*{:180}{$t^3$}
  \end{pspicture}.
\end{equation}
For example, the arrow
$
 \circlenode{1}{1} \hspace{12mm} \circlenode{0}{0}
   \ncline[linecolor=blue]{->}{1}{0} \ncput*{t_{\phantom{x}}^4}
$
from the vertex $\circlenode{1}{1} = \scO_\bK(\vecy)$
to the vertex $\circlenode{0}{0} = \scO_\bK$
comes from
\begin{align*}
 \Hom(\scO_\bX(\vecy), \scO_\bX(- 4 \vecomega))
  &= H^0(\scO_\bX(-\vecy + 4 (\vecx + \vecy + \vecz - \vecc))) \\
  &= H^0(\scO_\bX(4 \vecx + 3 \vecy + 4 \vecz - 4 \vecc))) \\
  &= H^0(\scO_\bX(2 \vecc + \vecc + \vecc - 4 \vecc))) \\
  &= H^0(\scO_\bX),
\end{align*}
and other new arrows are computed similarly.

As another example,
the quiver $\Qtilde$
for the Euclidean weighted projective line
$\bX_{3,3,3}$ is shown
(without labels) as follows:
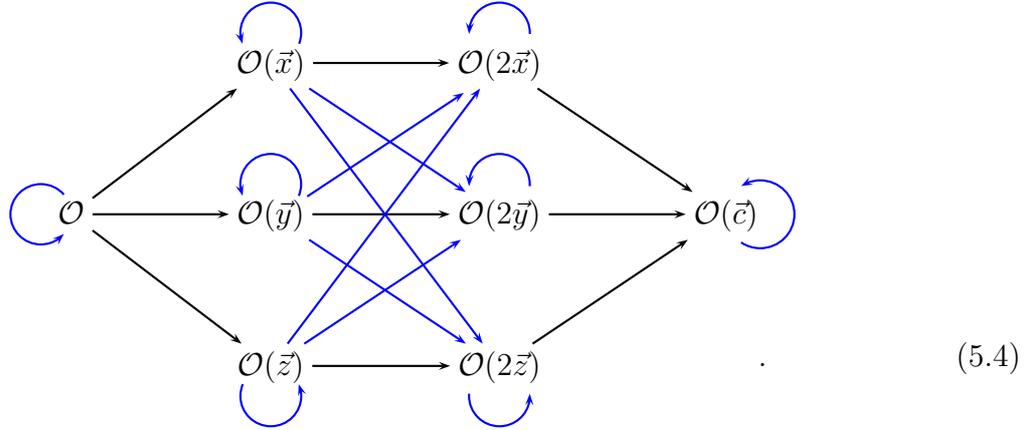
\vspace{5mm}
\begin{equation}\vspace{3mm}
\begin{psmatrix}[colsep=20mm]
 & \scO(\vecx) & \scO(2 \vecx) \\
 \scO & \scO(\vecy) & \scO(2 \vecy) & \scO(\vecc) \\
 & \scO(\vecz) & \scO(2 \vecz)
\end{psmatrix}.
\psset{nodesep=3pt,arrows=->,shortput=nab}
\ncline{2,1}{1,2}
\ncline{1,2}{1,3}
\ncline{1,3}{2,4}
\ncline{2,1}{2,2}
\ncline{2,2}{2,3}
\ncline{2,3}{2,4}
\ncline{2,1}{3,2}
\ncline{3,2}{3,3}
\ncline{3,3}{2,4}
\ncline[linecolor=blue]{1,2}{2,3}
\ncline[linecolor=blue]{1,2}{3,3}
\ncline[linecolor=blue,offset=-2pt]{2,2}{1,3}
\ncline[linecolor=blue]{2,2}{3,3}
\ncline[linecolor=blue]{3,2}{1,3}
\ncline[linecolor=blue]{3,2}{2,3}
\nccircle[angleA=0,angleB=180,ncurv=5,linecolor=blue,nodesep=0pt]{->}{1,2}{.4}
\nccircle[angleA=0,angleB=180,ncurv=5,linecolor=blue,nodesep=0pt]{->}{1,3}{.4}
\nccircle[angleA=90,angleB=270,ncurv=5,linecolor=blue,nodesep=3pt]{->}{2,1}{.4}
\nccircle[angleA=0,angleB=180,ncurv=5,linecolor=blue,nodesep=0pt]{->}{2,2}{.4}
\nccircle[angleA=0,angleB=180,ncurv=5,linecolor=blue,nodesep=0pt]{->}{2,3}{.4}
\nccircle[angleA=270,angleB=90,ncurv=5,linecolor=blue,nodesep=0pt]{->}{2,4}{.45}
\nccircle[angleA=180,angleB=0,ncurv=5,linecolor=blue,nodesep=0pt]{->}{3,2}{.4}
\nccircle[angleA=180,angleB=0,ncurv=5,linecolor=blue,nodesep=0pt]{->}{3,3}{.4}
\end{equation}
For example,
the arrow from
$\scO_\bK(\vecx)$ to $\scO_\bK(2 \vecy)$
is labeled as $t \cdot z$
and comes from
\begin{align*}
 \Hom(\scO_\bX(\vecx), \scO_\bX(2 \vecy - \vecomega))
  &= H^0(\scO_\bX(-\vecx + 2 \vecy + (\vecx + \vecy + \vecz - \vecc))) \\
  &= H^0(\scO_\bX(3 \vecy + \vecz - \vecc))) \\
  &= H^0(\scO_\bX(\vecc + \vecz - \vecc))) \\
  &= H^0(\scO_\bX(\vecz)),
\end{align*}
the arrow from $\scO$ to itself
is labeled as $t^3$
and comes from
\begin{align*}
 \Hom(\scO, \scO(3 \vecomega))
  &= H^0(\scO) \cong \bC,
\end{align*}
and similarly for all other new arrows.
Note that one has $3 \vecomega = 0$ in this case.

\begin{remark} \label{rm:Qtilde}
The quivers $Q$ and $\Qtilde$ are different when $n=3$, since one has
\begin{align*}
 \Hom(\scO_\bX(\vecx_i), \scO_\bX((p_j-1) \vecx_j - \vecomega))
  &= H^0(\scO_\bX(-\vecx_i + (p_j-1) \vecx_j
   + (\vecx_i + \vecx_j + \vecx_\ell - \vecc))) \\
  &= H^0(\scO_\bX(\vecx_\ell)))
\end{align*}
for mutually distinct $(i, j, \ell)$. Therefore one has a `new' arrow
labeled $t x_\ell$
from $\pi^*\scO_\bX(\vecx_i)$ to $\pi^*\scO_\bX((p_j-1)\vecx_j)$.

However, if $n \ge 4$, then one has
\begin{align*}
 \Hom(\scO_\bX, \scO_\bX(\vecc - \vecomega))
  &= H^0(\scO_\bX(\vecc + (\vecx_1 + \cdots + \vecx_n - (n-2) \vecc))) \\
  &= H^0(\scO_\bX(\vecx_1 + \cdots + \vecx_n - (n-3) \vecc)) \\
  &= 0,
\end{align*}
and $\Ext^\ell(E_i, E_j \otimes \scO_\bX(-k \vecomega))$ vanishes
similarly for any $i$, $j$, $\ell$ and positive $k$,
so the quiver $\Qtilde$ coincides with $Q$,
and the relation $\Itilde$ coincides with $I$.
\end{remark}

\section{The canonical bundle as a fine moduli space}
 \label{sc:moduli2}

Let $\bX$ be a weighted projective line
with at most three orbifold points,
so that the quiver $\Qtilde$ has the same set of vertices as $Q$
and a different set of arrows from $Q$ as explained in \pref{rm:Qtilde}.
The Cox ring $S_\bK$ is graded by $\Pic \bK$,
which is isomorphic to $L := \Pic \bX$
by the pull-back morphism $\pi^* : \Pic \bX \to \Pic \bK$.
Each arrow of $\Qtilde$ is naturally labeled
by a monomial in the Cox ring
as in \eqref{eq:quiver_K},
and one can define a {\em framed refined representation}
of $(\Qtilde, \Itilde)$
as follows:

\begin{definition}
A {\em framed refined representation} of $(\Qtilde, \Itilde)$
of dimension vector $\bsone$ consists of
\begin{itemize}
 \item
a representation
$((W_\bi)_{\bi \in Q_0}, (w_a)_{a \in \Qtilde_1})$
of the quiver $\Qtilde$ satisfying the relations $\Itilde$
of dimension vector $\bsone$,
 \item
one-dimensional vector spaces $V_x$, $V_y$, $V_z$, $V_t$,
 \item
elements $v_x \in V_x$, $v_y \in V_y$, $v_z \in V_z$, $v_t \in V_t$,
 \item
non-zero linear maps
$f_a : V_a \to \Hom(W_{s(a)}, W_{t(a)})$
for each $a \in \Qtilde_1$, and
 \item
a non-zero element $g \in W_0$
\end{itemize}
satisfying
\begin{equation} \label{eq:framing}
 f_a(v_a) = w_a 
\end{equation}
for all $a \in \Qtilde_1$.
\end{definition}
Here, the vector space $V_a$ is defined as
$
 V_x^{\otimes a_x} \otimes V_y^{\otimes a_y}
  \otimes V_z^{\otimes a_z} \otimes V_t^{\otimes a_t}
$
when the arrow $a$ corresponds to
the monomial $x^{a_x} y^{a_y} z^{a_z} t^{a_t}$
in $S_\bK$,
and the element $v_a \in V_a$ is defined similarly as
$
 v_x^{\otimes a_x} \otimes v_y^{\otimes a_y}
  \otimes v_z^{\otimes a_z} \otimes v_t^{\otimes a_t}.
$
Note that unlike the construction of Kronheimer
\cite{Kronheimer_ALE},
we frame arrows instead of vertices.
The category $\scM^\fr(\Qtilde, \Itilde; \bsone)$
of flat families of framed refined representations
of dimension vector $\bsone$ is defined
just the same way as in Definition \ref{df:M}.

Now we prove Theorem \ref{th:main2}.

\begin{proof}[Proof of Theorem \ref{th:main2}]
By choosing bases of $W_\bi$ and $V_i$,
a refined representation is represented by a point
in the subvariety $\Ztilde$ of
$$
 \bA^{\Qtilde_1}
  \times \bA^4
  \times (\bG_m)^{\Qtilde_1}
  \times \bG_m
$$
consisting of $((w_a)_a, (v_i)_i, (f_a)_a, g)$
satisfying the relations $\Itilde$
and \eqref{eq:framing}.
Two elements of $\Ztilde$ define isomorphic representations
if they lie in the same orbit of the group
$
 \bG_m^{Q_0}
$
under the change-of-basis action.
This shows that the category $\scM^\fr(\Qtilde, \Itilde; \bsone)$
is a stack isomorphic to $[\Ztilde / \bG_m^{Q_0}]$.

Let $\vartheta$ be the stability parameter
given in \eqref{eq:vartheta},
which defines a stability of framed refined representations 
of $(\Qtilde, \Itilde)$
just the same way as in Definition \ref{df:stability}.
The moduli space $\scM_\vartheta^\fr(\Qtilde, \Itilde; \bsone)$
of $\vartheta$-stable framed refined representations of $(\Qtilde, \Itilde)$
of dimension vector $\bsone$
is an open substack of $\scM^\fr(\Qtilde, \Itilde; \bsone)$,
which is isomorphic to the quotient stack
$
 [\Ztilde^s / \bG_m^{Q_0}]
$
of the subvariety $\Ztilde^s \subset \Ztilde$ consisting of $\vartheta$-stable framed representations;
\begin{align} \label{eq:K1}
 \scM_\vartheta^\fr(\Qtilde, \Itilde; \bsone)
  \cong [\Ztilde^s / \bG_m^{Q_0}].
\end{align}
Every point in $\Ztilde$ lies in the orbit of a point
whose $\bG_m^{\Qtilde_1} \times \bG_m$-component
is given by $(1, \ldots, 1)$.
Such a point is determined by its $\bA^4$-component,
since the $\bA^{\Qtilde_1}$-component is determined
from the $\bA^4$-component
by \eqref{eq:framing}.
A point with the $\bA^4$-component
$(v_x, v_y, v_z, v_t)$
satisfies the relations $\Itilde$
if and only if the relation
$$
 v_x^{p_1} - v_y^{p_2} + v_z^{p_3} = 0
$$
is satisfied.
This shows that
\begin{align} \label{eq:K2}
 \Ztilde \cap
  \lb \bA^{\Qtilde_1} \times
   \Spec \bA^4 \times \{ (1, \ldots, 1) \} \rb
  \cong \Spec S_\bK.
\end{align}
The characters of the subgroup of $\bG_m^{Q_0}$
fixing the $\bG_m^{\Qtilde_1} \times \bG_m$-component
is given by $\bZ^{Q_0} / K \cong L$
just as in the case of $Q$.
Unstable locus consists of
\begin{align} \label{eq:K3}
 \bszero \times \Spec \bC[t]
  \subset \Spec S_\bX \times \Spec \bC[t] = \Spec S_\bK,
\end{align}
so that the moduli space is given by
\begin{align} \label{eq:K4}
 [((\Spec S_\bX \setminus \bszero) \times \Spec \bC[t]) / G].
\end{align}
Now \eqref{eq:Cox_ring2} shows that
the moduli space is isomorphic to $\bK$,
and the tautological line bundles are given by $\pi^* E_i$.
\end{proof}

\bibliographystyle{amsalpha}
\bibliography{bibs}

\noindent
Tarig Abdelgadir 

Mathematics Section,
The Abdus Salam International Centre for Theoretical Physics,
Strada Costiera 11,
I - 34151 Trieste,
Italy.

{\em e-mail address}\ : \  tarig.m.h.abdelgadir@gmail.com

\ \vspace{0mm} \\

\noindent
Kazushi Ueda

Department of Mathematics,
Graduate School of Science,
Osaka University,
Machikaneyama 1-1,
Toyonaka,
Osaka,
560-0043,
Japan.

{\em e-mail address}\ : \  kazushi@math.sci.osaka-u.ac.jp
\ \vspace{0mm} \\

\end{document}